\RequirePackage{fix-cm}
\documentclass[smallextended]{svjour3}       
\smartqed  
\usepackage{graphicx}
\usepackage{amssymb}
\usepackage{amsmath}      
\usepackage{latexsym}
\usepackage{amsfonts}
\usepackage{multicol}

\DeclareMathOperator{\ima}{{Im}}
\DeclareMathOperator{\End}{{End}}

\newcommand{\cA}{\mathcal{A}}

\newcommand{\cW}{\mathcal{W}}

\newcommand{\RR}{\mathbb R}
\newcommand{\QQ}{\mathbb Q}

\newcommand{\ZZ}{\mathbb Z}
\newcommand{\CC}{\mathbb C}

\begin{document}

\title{On the group algebra decomposition of a Jacobian variety\thanks{
Partially supported by Fondecyt Grant 1100113 and Conicyt Fellowship for PhD. studies}
}

\titlerunning{On the group algebra decomposition of JX}        

\author{Leslie Jim\'enez}


\institute{Departamento de Matem\'aticas, Facultad de Ciancias, Universidad de Chile, Santiago \at
              Las palmeras 3425, \~Nu\~noa \\
              Tel.: +56-2-29787288\\
              \email{leslie.jimenez@liu.se}\\
              \emph{Present address: Link\"{o}pings universitet, 581 83 LINK\"{O}PING.}  
}

\date{Received: date / Accepted: date}

\maketitle

\begin{abstract}
Given a compact Riemann surface $X$ with an action of a finite group $G$, the group algebra $\QQ[G]$ provides an isogenous decomposition of its Jacobian variety $JX$, known as the group algebra decomposition of $JX$. We obtain a method to concretely build a decomposition of this kind. Our method allows us to study the geometry of the decomposition. For instance, we build several decompositions in order to determine which one has kernel of smallest order. We apply this method to families of trigonal curves up to genus $10$.
	
\end{abstract}

\section{Introduction}\label{intro}
The action of a finite group on a given compact Riemann surface $X$ of genus $g\geq 2$ induces a homomorphism $\rho:\QQ[G]\to \text{End}_{\QQ}(JX)$ from the rational group algebra $\QQ[G]$ into the rational endomorphism algebra of $JX$ in a natural way. The factorization of $\QQ[G]$ into a product of simple algebras yields a decomposition of $JX$ into abelian subvarieties \cite{l-r}, \cite{kr}, up to isogeny.  

This decomposition, and in general Jacobians with group action, have been extensively studied from different points of view \cite{serre}, \cite{l-r}, \cite{lro}, \cite{paulhus}, \cite{paulhus2}, \cite{lrro}, \cite{brr}, \cite{cr}, \cite{sevin}, \cite{yoibero}, \cite{kr}, \cite{accola1}.  

In \cite{paulhus2} the decomposition of Jacobians of hyperelliptic curves in elliptic factors is studied. The author developed a nice geometrical description for these Jacobians up to genus $10$. Her motivation came from \cite{serre} where they asked for completely decomposable Jacobians of any dimension $g$ (They gave examples up to $g=1297$). In \cite{yoibero} an explicit formula to calculate the dimension of the factors in the decomposition of $JX$ is given and the polarizations of the subvarieties in the decomposition of $JX$ are studied in \cite{lro} and \cite{lrro}. In general, the kernel of the decomposition of $JX$ has not been studied. The only references treating kernels we know are: \cite{reci-ro} for Jacobians with action of the symmetric group of order 3, \cite{naka} where the author studies families of curves whose Jacobians are isomorphic to a product of elliptic curves and \cite{crro} where the authors study dihedral actions on Jacobians, but the tools used to compute kernels are different from the method developed here.

In this work, we present a method to concretely build an isogeny which is a group algebra decomposition (Section \ref{factoresjacobianas}). We do this deepening the method developed in \cite{lro}. This allows us to describe the lattices of the factors in a group algebra decomposition. Moreover, we find a method to determine the order of the kernel of this isogeny. We give a specific criterion to choose the subvarieties in a group algebra decomposition having a kernel of smallest possible order. In particular, we can decide when the isogeny is an isomorphism. 

We apply our method to non-normal trigonal curves (see \cite{w1}) and normal trigonal curves of genus $g < 10$ with reduced group $A_4, S_4$ and $A_5$. In these cases, we determine the factors of the decomposition such that the size of the kernel is the minimum possible (Section \ref{s:results}). The description of the lattices of the factors in the decomposition of $JX$ provides an alternative explanation of the subspace of the loci of Jacobians of trigonal curves inside the moduli space $\cA_g$ of principally polarized abelian varieties by means of the Riemann matrix of the decomposition (Remark \ref{moduli}).

\section{Preliminaries}\label{prelimi}

Let $G$ be a finite group. The known results about the representations of $G$ used in this section may be founded in \cite{serre1} and \cite{curtis}. 

If $V$ is an irreducible representation of $G$ over $\CC$, we denote $F$ as its field of definition and $K$ the field obtained by extending $\QQ$ by the values of the character $\chi_V$; then $K\subseteq F$ and $m_V=[F:K]$ is the Schur index of $V$. 

If $H$ is a subgroup of $G$, $\text{Ind}_{1_H^G}$ will denote the representation of $G$ induced by the trivial representation of $H$ and $\langle U,V \rangle$ denotes the usual inner product of the characters. By the Frobenius Reciprocity Theorem $\langle \text{Ind}_{1_H^G},V\rangle=\text{dim}_\CC V^H$, where $V^H$ is the subspace of $V$ fixed by $H$.

Any compact Riemann surface $X$ of genus $g$ has an associated principally polarized abelian variety $JX$ (i.e. a complex tori with a principal polarization). This variety is called \textit{the Jacobian variety of $X$} and has complex dimension $g$. Good accounts of abelian varieties and Jacobians are given in \cite{bl} and \cite{rubi}.

Given a compact Riemann surface $X$ with an action of a group $G$, we consider the induced homomorphism  $\rho:\QQ[G]\to \text{End}_\QQ(JX)$. For any element $\alpha \in \QQ[G]$ we define an abelian subvariety
\begin{equation}\label{B-alfa}
B_{\alpha} := \ima (\alpha)=\rho(l\alpha)(JX) \subset JX,
\end{equation}
where $l$ is some positive integer such that $l\alpha \in \ZZ[G]$. 

As $\QQ[G]$ decomposes into a product $Q_0\times \dots \times Q_r$ of simple $\QQ-$algebras, the simple algebras $Q_i$ are in bijective correspondence with the rational irreducible representations of $G$. That is, for any rational irreducible representation $\cW_i$ of $G$ there is a uniquely determined central idempotent $e_i$. This idempotent defines an abelian subvariety, namely $B_i=B_{e_i}$. These varieties, called isotypical components, are uniquely determined by the representation $\cW_i$.

The addition map is an isogeny \cite[section 2]{l-r}.
\begin{equation}\label{iso}
\mu:B_0\times \dots \times B_r\to JX, 
\end{equation}
which is called \textit{the isotypical decomposition of $JX$}. 

Moreover, the decomposition of every $Q_i=L_1\times \dots \times L_{n_i}$ into a product of minimal left ideals (all isomorphic) gives a further decomposition of the Jacobian. There are idempotents $f_{i1},\dots, f_{in_i}\in Q_i$ such that $e_i=f_{i1}+\dots +f_{in_i}$ where $n_i=\text{dim} V_i/m_{V_i}$, with $V_i$ the $\CC-$irreducible representation associated to $\cW_i$. These idempotents provide subvarieties $B_{ij}:=B_{f_{ij}}$ of $JX$. 

It is known that the factor in the isotypical decomposition of $JX$ associated to the trivial representation of $G$ is isogenous to $J_G=J(X/G)$. This factor will be denoted by $B_0$. Then the addition map is an isogeny 

\begin{equation}\label{langerecillas}
\nu:J_G\times \Pi_1^{n_1}B_{1j}\times \dots \times \Pi_1^{n_r}B_{rj}\to JX.
\end{equation}

This is called \textit{the group algebra decomposition of $JX$} \cite{lrro}. We use this name to refer to the isogeny $\nu$ as well.
Note that since all the minimal left ideals decomposing $Q_i$ are isomorphic, which implies that the sub-varieties defined by the idempotents $f_{ij}$ are isogenous, we may write the group algebra decomposition as

\begin{equation}\label{nu-tilde}
\tilde{\nu}:J_G\times B_{11}^{n_1}\times \dots \times B_{r1}^{n_r}\to JX,
\end{equation}
which is the classical way of writing it. The problem with this is that one of our goals is to minimize the order of the kernel of $\nu$, and there are examples (for instance \cite{kft}) where we may even obtain isomorphisms by changing the components in the same isogeny class.
Therefore, we will stay with the decomposition (\ref{langerecillas}) because it will allow us to reduce the kernel of the isogeny. 

If two complex tori $B$ and $B'$ are isogenous, we write $B\sim B'$.

For group actions on a Riemann surface we follow the notation and definitions given in \cite{fk}. We define the group of automorphisms $\text{Aut}(X)$ of a Riemann surface $X$ as the analytical automorphism group of $X$. We say that a finite group $G$ acts on $X$ if $G\leq \text{Aut}(X)$. 
The quotient $X/G$ (the space of the orbits of the action of $G$ on $X$) is a compact Riemann surface with complex atlas given by the holomorphic branched covering $\pi_G: X\to X/G$. The degree of $\pi_G$ is $|G|$ and the multiplicity of $\pi_G$ at $p$ is $\text{mult}_p(\pi_G)=|G_p|$ for all $p\in X$, where $G_p$ denotes the stabilizer of $p$ in $G$. If $|G_p|\neq 1$ then $p$ will be a branch point of $\pi_G$.

Let $\{p_1,...,p_r\}\subset X$ be a maximal collection of non-equivalent branch points with respect to $G$. We will denote $\gamma=g(X/G)$.   
We define the \textit{the signature (or branching
data) of G on X} as the vector of numbers $(\gamma; m_1,...,m_r)$ where $m_i=|G_{p_i}|$. We have the Riemann-Hurwitz formula $g(X)=|G|(\gamma-1)+1+\frac{|G|}{2}\sum_{i=1}^{r}\big(1-\frac{1}{m_i}\big)$, where $g(X)$ denotes the genus of $X$. 

A $2\gamma+r$ tuple $(a_1, \dots, a_{\gamma},b_1, \dots,
b_{\gamma},c_1, \dots,c_r)$ of elements of $G$ is called a
\textit{generating vector of type $(\gamma;m_1,...,m_r)$} if the
following are satisfied:

i) $G$ is generated by the elements $(a_1, \dots, a_{\gamma},b_1,
\dots, b_{\gamma},c_1, \dots,c_r)$;

ii) $\text{order}(c_i)=m_i$; and

iii) $\prod_{i=1}^{\gamma}[a_i,b_i]\prod_{j=1}^rc_j=1$, where $[a_i,b_i]$ is the commutator of $a_i, b_i \in G$. 

The existence of a generating vector of given type ensures the existence of a Riemann surface with an action of a given finite group.  The dimension of the subvarieties in the decomposition (\ref{langerecillas}) are obtained using the generating vector of the action \cite[Theorem 5.12]{yoibero}.

Moreover, the induced action of $G$ on $JX=\CC^n/\text{H}_1(X,\ZZ)$ provides geometrical information about the components of the group algebra decomposition of $JX$ \cite{cr}. 

\begin{definition}\label{ph}
For any subgroup $H$ of $G$, define $p_H=\frac{1}{\vert H\vert}\sum_{h\in H} h$ as the central idempotent in $\QQ[H]$ corresponding to the trivial representation of $H$. Also, we define $f_H^{i}$ as $p_H e_{i}$, an idempotent element in $\QQ[G]e_{i}$.
\end{definition}

Then the corresponding group algebra decomposition of $J_H=J(X/H)$ is given as follows \cite[Proposition 5.2]{cr}:
\begin{equation}\label{carocca-rodriguez}
J_H \sim J_G\times B_{11}^{\frac{\text{dim} V_1^H}{m_1}}\times \dots \times B_{r1}^{\frac{\text{dim} V_r^H}{m_r}}.
\end{equation}
Moreover, 
\begin{equation}\label{phjxh}
\ima (p_H)=\pi_H^*(J_H)
\end{equation}
where $\pi_H^*(J_H)$ is the pullback of $J_H$ by $\pi_H$ (see \cite{bl}). If $\text{dim} V_i^H\neq 0$ then
\begin{equation}\label{fh}
\ima (f_H^i)=B_{i1}^{\frac{\text{dim} V_i^H}{m_{V_i}}}.
\end{equation}


\section{Describing the factors via a symplectic representation}\label{S:method}

We are interested in describing the factors of the group algebra decomposition of a Jacobian variety with group action given in the previous section.

The method we follow is to describe the lattice of such factors. We apply \cite[Section 2]{lro}, but extended to any symmetric idempotent $\alpha=\sum_{g\in G}a_gg$. 

Let $\Lambda=H_1(X,\ZZ)$ denote the lattice of $JX$ and $\Lambda_{\QQ}=\Lambda\otimes_{\ZZ} \QQ$. Let $\rho_s:\QQ[G]\to \End(\Lambda_{\QQ})$ denote the morphism induced by the (symplectic) rational representation of $G$ in $\Lambda$, $\rho:\QQ[G]\to \End_{\QQ}J X$ is the homomorphism given by the action of the group $G$ on $X$ and $\rho_r$ is the rational representation of $\text{End}_\QQ(JX)$ which completes the diagram below. 

\begin{center}
\newcounter{cms}
\setlength{\unitlength}{1cm}
\begin{picture}(7,7.5)\label{diag1}
\put(3.8,3.2){\vector(0,1){3.3}}
\put(2,4.9){\vector(1,1){1.6}}
\put(2,4.8){\vector(1,-1){1.7}}
\put(1.4,4.7){\makebox(0,0)[b]{$\QQ[G]$}}
\put(3.9,6.6){\makebox(0,0)[b]{$\End(\Lambda_{\QQ})$}}
\put(3.9,2.5){\makebox(0,0)[b]{$\End_\QQ(J X)$}}
\put(2.5,6.1){\makebox(0,0)[tl]{$\rho_s$}}
\put(2.5,3.6){\makebox(0,0)[b]{$\rho$}}
\put(4,5){\makebox(0,0)[tl]{$\rho_r$}}
\end{picture}
\end{center}
\vspace{-2cm}

Thus $\rho_r$ induces a rational representation of $\text{Hom}_\QQ(\Pi_{ij}B_{ij}, JX)$ given by $\nu\to \nu_{\Lambda}:\oplus_{ij} \Lambda_{ij}\to \Lambda,$
where $\nu_{\Lambda}$ is the restriction of $\nu$ to the lattice $\oplus_{ij} \Lambda_{ij}$ of the product of the $B_{ij}$ in $\nu$. 

Then $\rho_s(\alpha)\in \End\Lambda_{\QQ}$ is given by $\rho_s(\alpha)=\sum_{g\in G}a_g\rho_s(g).$

Therefore we have the following facts, analogous to \cite[section 2]{lro}.

\begin{proposition}\label{lattice}
Let $\alpha\in \QQ[G]$. The sublattice of $\Lambda$ defining $B_{\alpha}=V_{\alpha}/\Lambda_{\alpha}$ is given by

$$\Lambda_{\alpha}:=\rho_s (\alpha)(\Lambda_{\QQ})\cap \Lambda,$$
where the intersection is taken in $\Lambda_{\QQ}$ and $\rho_s (\alpha)$ is the image of $\alpha$ by $\rho_s$. In this case, the $\CC-$vector space $V_{\alpha}$ is generated by $\rho_s (\alpha)(\Lambda_{\QQ})$.
\end{proposition}

The previous construction is clearer in its matrix form, once bases are chosen. From here to the end of this work, we use this form for determining the lattice of the factors in $\nu$.

\begin{remark}\label{obs-lattice}
Suppose we have $\Gamma=\{ \alpha_1, \dots, \alpha_{2g}\}$ a (symplectic) basis of the lattice $\Lambda$ of the Jacobian $JX$ (assumed to be of dimension $g$). Consider  $\rho_s$ in its matrix form (with respect to this basis). 

For any $g\in G$, we have that $\rho_s(g)$ is a square matrix of size $2g$. Hence, for any $\alpha \in \QQ[G]$, we have associated to it a rational $2g \times 2g$-matrix 
$$M = (m_{ij}).$$
The $j$-th column of $M$ corresponds to the element $\rho_s(g)(\alpha_j)=\sum_{i=1}^{2g} m_{ij} \alpha_i$ and the lattice $\Lambda_{\alpha}$ of $B_{\alpha}$ corresponds to

$$\Lambda_{\alpha} := (\langle M \rangle_{\ZZ} \otimes \QQ) \cap \Lambda,$$
where $\langle M \rangle_{\ZZ}$ denotes the lattice over $\ZZ$ generated by the columns of $M$. 

In other words, the lattice $\Lambda_{\alpha}$ is obtained by considering the $\RR$-linearly independent columns of $M$ and intersecting it with $\Lambda$. Our next step is to look for a basis of $\Lambda_{\alpha}$, which will be in terms of the elements of $\Gamma$. By computing its coordinates in the basis $\Gamma$, we get a $2g\times 2\dim B_{\alpha}$ coordinate matrix of the lattice $\Lambda_{\alpha}$. Moreover, $V_{\alpha}$ is the complex vector space generated by the column vectors of $M$.
\end{remark}

\section[Method]{Method. A group algebra decomposition $\nu_{\times}$}\label{factoresjacobianas}

In this section we develop the core of this work. We present a method to concretely build an isogeny $\nu$ as in \eqref{langerecillas}. 

Given a compact Riemann surface $X$ with the action of a group $G$, the general theory gives us the existence of a group algebra decomposition for the corresponding Jacobian variety $JX$. The results in \cite{yoibero} allow us to compute the dimensions of the factors. Nevertheless to describe further geometrical properties such as induced polarization, period matrix, etc., we need an explicit description of the factors. Section \ref{S:method} gives us a tool to solve some of these questions, under certain hypotheses for $G$. 

We present here a method to find a set of primitive idempotents $f_{i1},\dots, f_{in_i}$ to describe the factors in this decomposition, in order to extract properties of the decomposition. This concrete construction will allow us to easily compute the order of the kernel of the isogeny $\nu$. Hence we may choose an \textit{optimal set} of those idempotents in the sense of getting the smallest possible kernel.

We consider the quotient of $X$ for the action of $G$ of genus $0$ for simplicity, because it is known that the factor in $\nu$ corresponding to the trivial representation is the image of $p_G$, hence it is isogenous to $J_G$.

\vskip12pt

\noindent{\textbf{Data:}} Let $X$ be a Riemann surface of genus $g\geq 2$ with the action of a group $G$ with total quotient of genus $0$. Assume that the symplectic representation $\rho_s$ for this action is known.
\vskip12pt

\textbf{1. STEP ONE:} Identification of factors using Jacobians of intermediate coverings.\\

The following lemma gives us conditions under which a factor in the group algebra decomposition can be described as the image of a concrete idempotent, in particular when it corresponds to a Jacobian of an intermediate quotient.

\begin{lemma}\label{metodo}
Let $X$ be a Riemann surface with an action of a finite group $G$ such that the genus of $X_G$ is equal to zero. Consider $\nu$, the group algebra decomposition of $JX$ as in \eqref{langerecillas}.  
\begin{itemize}
\item[(i)] If $H\leq G$ is such that $\dim_\CC V_i^{H}=m_i$, where $m_i$ is the Schur index of the representation $V_i$, then for some $j\in\{1, ..., n_{i}\}$ we have that
$$\ima (f_H^i)=B_{ij}.$$
In addition, 
\item[(ii)] if $\dim_\CC V_l^{H} = 0$ for all $l$, $l\neq i$, 
such that $\dim_\CC B_l\neq 0$ in the isotypical decomposition of $JX$ in Equation (\ref{iso}), then 
$$J_H\sim \ima (p_H)=B_{ij}.$$
\end{itemize}
\end{lemma}

\begin{proof}
From Proposition \ref{lattice} and the fact that $\dim_\CC V_i^{H}\neq 0$, we get  $\ima (f_H^{i}) = \Pi_i^{k_i}B_{ij}$. Since $\dim_\CC V_i^{H}=m_i$ (by hypothesis), we obtain that $k_i=1$. Hence $\ima (f_H^{i}):=B_{ij}$ for some $j\in \{1, ..., n_{i}\}$.

If, in addition, $\dim_\CC V_l^{H} = 0$ (equivalently $\dim_\QQ \cW_l^H=0$) for all $l\neq i$ such that $\dim_\CC B_l\neq 0$, then $f_H^l=0$ for all of them. Due to the fact that $p_H=\sum_{l\in\{1,...,r\}} f_H^l$ we obtain that $p_H=f_H^i$. Moreover, by equation (\ref{carocca-rodriguez}) we have $J_H\sim \ima (p_H)=B_{ij}.$
\end{proof}

\begin{remark}\label{hconjugate}
Observe that if $H$ satisfies Lemma \ref{metodo} for some $i\in \{1,...,r\}$, then all its conjugates satisfy it for the same $i$. To see this, consider $H'=H^g=gHg^{-1}$. It is clear that for all $s\in G$, we have
 
\begin{eqnarray*}
\dim_\CC V_i^H&=&\langle\text{Ind}_{1_H^G}, V_i\rangle=\frac{1}{|G|}\sum_{t\in G}\chi_{\text{Ind}_{1_H^G}}(t)\chi_{V_i}(t^{-1})\\
&=&\frac{1}{|G|}\sum_{t\in G}\Big(\sum_{s^{-1}ts\in H}\chi_{1_H}(s^{-1}ts)\Big)\chi_{V_i}(t^{-1})\\
&=&\frac{1}{|G|}\sum_{t\in G}|H|\chi_{V_i}(t^{-1})=\frac{1}{|G|}\sum_{t\in G}|H'|\chi_{V_i}(t^{-1})\\
&=&\frac{1}{|G|}\sum_{t\in G}\Big(\sum_{s^{-1}ts\in H'}\chi_{1_{H'}}(s^{-1}ts)\Big)\chi_{V_i}(t^{-1})\\
&=&\dim_\CC V_i^{H'}.\\
\end{eqnarray*} 
\end{remark}

Our purpose is to use this result \textit{conversely} to actually produce an isogeny $\nu$.

\vskip12pt

\textbf{STEP TWO:} Definition of certain subvarieties of $JX$.



\vskip12pt

\begin{definition}\label{D:B_ij}
Let $H$ be a subgroup of $G$ satisfying condition (i) of Lemma \ref{metodo} for some $i\in \{1,...,r\}$, then define $B_H$ as the image of $f_H=p_He_i$.
\end{definition}

Note that depending on the geometry of the action, $B_H$ can be trivial.

From Proposition \ref{lattice}, we get that its lattice corresponds to

\begin{equation}\label{latticeph}
\Lambda_{H} := (\langle \rho_s (f_H) \rangle_{\ZZ} \otimes \QQ) \cap H_1(X,\ZZ).
\end{equation}


Using the procedure described in Remark \ref{obs-lattice}, we obtain the coordinate matrix corresponding to a basis of the lattice of $B_{H}$. We sometimes use the same symbol $\Lambda_H$ to denote this matrix.

\begin{corollary}
Let $H$ be a subgroup of $G$ satisfying both conditions of Lemma \ref{metodo}. Then $f_H=p_He_i=p_H$, and $B_H$ is isogenous to the Jacobian of $X/H$.

\end{corollary}
\begin{proof}

By Lemma \ref{metodo}, if $H$ satisfies both conditions then the Jacobian of $X/H$ is isogenous to one of the factors in a group algebra decomposition for $JX$. This is equivalent to the equality of their corresponding idempotents.
\end{proof}




\vskip12pt

\textbf{3. STEP THREE:} Construction of a product subvariety $B_{\times}$ of $JX$.

\vskip12pt

If we have \textit{enough} subgroups from STEP TWO, we may construct the product of all the subvarieties defined by those subgroups. This will be a subvariety of $JX$, its lattice is described in the following definition.

\begin{definition}\label{matrizLph}
Let $r+1$ be the number of rational irreducible representations of $G$. Suppose for all $i\in \{1,\dots,r\}$ and all $j=\{1,\dots, n_i\}$ there is a subgroup $H_{ij}$ satisfying condition (i) of Lemma \ref{metodo}. For each $i,j$ take one $H_{ij}$, and let 
$$S=\{H_{ij}:i\in \{1,\dots, r\},j\in \{1,\dots, n_i\}\}$$
be the set of these subgroups, where we do not consider subgroups $H_{ij}$ such that $\ima(f_{H_{ij}})=0$.
We define $L_{S} \in M_{2g}(\ZZ)$ to be the coordinate matrix given by the vertical join of the coordinate matrices of the lattices $\Lambda_{H_{ij}}$ (see equation \ref{latticeph}) for $H_{ij}\in S$.

We recall here that $i=0$ corresponds to the trivial representation whose factor is not considered here, $n_i=\dim V_i/m_i$, where $m_i$ is the Schur index of a complex irreducible representation $V_i$ associated to the rational irreducible representation corresponding to the factor $B_{ij}^{n_i}$ (from \eqref{langerecillas}).
\end{definition}

The lattice $\Lambda_{\times}$ defined by the matrix $L_{S}$, corresponds to the sublattice $\oplus_{ij} \Lambda_{H_{ij}}$ of $\Lambda=\text{H}_1(X,\ZZ)$. It is the lattice of the following subvariety of $JX$ 

$$B_{\times}:=\Pi_{i,j} B_{H_{ij}},$$
where $B_{H_{ij}}$ is as in Definition \ref{D:B_ij}.

\vskip12pt

\begin{definition}\label{D:isogenyLeslie}
With the above notation define a sum map $\nu_{\times}:B_{\times}\to JX$ 

$$\nu_{\times}(b_{11},\dots, b_{rn_r})=\sum_{i,j}b_{ij}\in JX.$$
\end{definition}

\vskip12pt

\textbf{STEP FOUR:} Condition for $\nu_{\times}$ to be an isogeny.

\vskip12pt

\begin{definition}\label{effectiveset}
Let $S=\{H_{ij}:i=1..r, j=1..n_i\}$ be a set of subgroups of $G$ as in Definition \ref{matrizLph}. We say that $S$ is an \textit{effective set for $G$} if the determinant of the corresponding matrix $L_{S}$ is different from $0$.  
\end{definition}

\begin{theorem}\label{T:isogeny} 
Let $X$ be a Riemann surface of genus $g\geq 2$ with an action of a group $G$ with total quotient of genus $0$.
Let $S=\{H_{ij}\}_{ij}$ be an effective set for $G$, then the map $\nu_{\times}$ definition \ref{D:isogenyLeslie} is an isogeny with kernel of order $|\det(L_{S})|$.
\end{theorem}

\begin{proof}
As before, denote by $\Lambda$ the lattice of $JX$. The map $\nu_{\times}$ induces a homomorphism of $\ZZ-$modules $\nu_{\Lambda}:\Lambda_{\times} \to \Lambda.$ If the rank of $\Lambda_{\times}$ is $2g$, then $\nu_{\Lambda}$ is a monomorphism of lattices in $\CC^g$. Moreover, all the sublattices $\Lambda_{H_{ij}}$ decomposing $\Lambda_{\times}$ correspond to subvarieties $B_{H_{ij}}$. Therefore the dimension of $B_{\times}$ is $g$.

It remains to show either $\nu_{\times}$ is surjective or its kernel is of finite order. For any isogeny $f: A_1\rightarrow A_2$ between two abelian varieties, it is known \cite[section 1.2]{bl} that

\begin{equation}\label{def-ker-matrix}
|\text{Ker}(f)| = \det \rho_r(f),
\end{equation}
where $\rho_r(f)$ is the rational representation of the isogeny $f$. It is known that if the kernel is finite, then its cardinality equals the index  
$[\Lambda:\Lambda_{\times}]$. As the matrix $L_S$ is non singular, we have that $\Lambda_{\times}$ is a lattice in $\CC^g$, hence this index is finite. After columns operations, the matrix $L_{S}$ is the matrix of the rational representation $\nu_{\times}$, therefore its kernel has order the absolute value of its determinant. 
\end{proof}

\begin{corollary}\label{isomorfismo}
Under the hypothesis in Theorem \ref{T:isogeny}, the isogeny $\nu_{\times}$ is an isomorphism if and only if $\det(L_{\{H_{ij}\}})=\pm1$. 
\end{corollary}

\begin{remark}
\begin{enumerate}

\item The isogeny from Theorem \ref{T:isogeny} corresponds to a group algebra decomposition isogeny $\nu$ as in \eqref{langerecillas}.



\item We point out that the isogeny $\nu_{\times}$ depends on the choice of the subgroups $H_{ij}$ for the set $S$. Therefore, its kernel may change if we change these subgroups. Our purpose is to move along different effective sets in order to achieve the smallest possible.
\end{enumerate}
\end{remark}

In the spirit of moving along different effective sets to minimize the order of the kernel, we have the following proposition (see \cite{yo} for a proof and more details).

\begin{proposition}\label{ph-conjugate}
If $H_2=gH_1g^{-1}$ for some element $g\in G$, then
\begin{itemize}
\item[(i)] $p_{H_2}=gp_{H_1}g^{-1}$,
\item[(ii)] $\ima(p_{H_2})$ is isomorphic to $\ima(p_{H_1})$. 
\end{itemize}
\end{proposition}


\begin{theorem}\label{pol}
Let $X$ be a Riemann surface of genus $g$ with an action of a finite group $G$ such that the genus of $X/G$ is zero. If $S=\{H_{i1},...,H_{in_i}\}$ is an effective set for $\nu$, then the induced polarization $E_{ij}$ of the factor $B_{ij}$ in $\nu$ is given by 
\begin{equation}\label{pola}
E_{ij} = \Lambda_{H_{ij}} E_{J X} \Lambda_{H_{ij}}^{t}.
\end{equation}
Moreover, the type of the polarization $E_{ij}$ is given by the elementary divisors of the matrix product (\ref{pola}). 
\end{theorem} 

\begin{remark}\label{O:filosofia}
Finally we point out that the principle idea of our method is to move along isomorphic varieties $B_H$, choosing different subgroups $H$ (even in the same conjugacy class) to construct the effective set $S$ and the corresponding isogenies $\nu_\times$, which may have kernels of different orders. This shows that the geometry of the varieties $B_H$ as subvarieties of $JX$ makes a difference. In fact the examples suggest that this reflects the way the subvarieties intersect each other. 
\end{remark}


\section{Application to trigonal curves}\label{s:results}

In this section we use the method explained in Section \ref{S:method}, the computational program MAGMA \cite{magma} and the algorithm introduced in \cite{brr}. To obtain the size of the kernel we use the full automorphism group of the curves. We study the Jacobians of families of trigonal curves up to genus $g < 10$. We divide our examples according to the genus (or the dimension) $g$. 

\subsection{Known facts about $3-$gonal curves and their automorphisms}\label{S:kr}

A compact Riemann Surface $X$ admitting a cyclic prime group of automorphisms $C_3$ of order $3$ such that $X/C_3$ has genus $0$ is called a cyclic $3-$gonal surface or a \textit{trigonal curve}. The group $C_3$ is called a $3-$gonal group for $X$. If $C_3$ is normal in the full automorphism group $\text{Aut}(X)$ of $X$, we call $X$ a normal cyclic $3-$gonal surface or a \textit{normal trigonal curve}. In this case, there exist only one $3-$gonal group for $X$. The quotient group $N_{\text{Aut}(x)}(C_3)/C_3$ between the normalizer $N_{\text{Aut}(x)}(C_3)$ of $C_3\leq\text{Aut}(X)$ and $C_3$ is called \textit{the reduced group of X}. It is well known \cite[section IV.9.3]{fk}, \cite[Table 1]{w2} that the reduced group of $X$ can be the cyclic group $C_n$ of order $n$, the dihedral group $D_n$ of order $2n$, the symmetric group $S_4$, the alternating groups $A_4$ or $A_5$.

A consequence of \cite[Lemma 2.1]{accola1} is the following result:

\begin{lemma}\label{normalg5}
If $X$ is a trigonal surface of genus $g\geq 5$, then $X$ is a normal trigonal curve.
\end{lemma}

The previous result allows us to find a list of automorphism groups of trigonal curves (see also \cite{mila-auto}). This list is obtained by an easy combination of \cite[Table 7]{w1}, \cite[Table 1]{buja} and \cite{magaard}, plus the computation of the reduced groups which are not in the original tables. We group the results in Tables \ref{T:1} and \ref{T:2}. Table \ref{T:1}, corresponds to non-normal trigonal curves. $\text{CD}$ denotes the central diagonal subgroup of $SL(2,3)$ of order $2$. Table \ref{T:2}, corresponds to normal trigonal curves with reduced group $A_4, S_4$ or $A_5$ \cite{buja}. We restrict to these reduced groups in the normal case mainly because the results in \cite{paulhus2} suggest that these families may have completely decomposable Jacobians (at least for some dimensions of $JX$). From Table \ref{T:2} we obtain the result that any trigonal curve with reduced group $S_{4},A_4$ or $A_5$ has even genus, and that there do not exist trigonal curves of genus 8 with reduced group $A_4, S_4$ or $A_5$.   

Note that in both tables we use the MAGMA notation ID to label the automorphism groups. This is denoted by a ordered pair which in the first entry is the size of the group and in the second entry is which group in MAGMA database it is.

\begin{table}
\begin{tabular}{|l|c|c|c|r|}\hline
Red. group & Automorphism group & ID & Genus & Signature \\ \hline
$D_2$ & $GL(2,3)$&(48,29)& 2 & (0;2,3,8)\\
$C_4$ & $SL(2,3)/CD$&(48,33)&3 & (0;2,3,12)\\
$D_3$ & $D_3\times D_3$&(36,10)&4 & (0;2,2,2,3)\\
$D_6$ & $(C_3\times C_3)\rtimes D_4$ & (72,40)& 4 & (0;2,4,6) \\ \hline
\end{tabular}

\vskip6pt
\caption{Full automorphism group for non-normal trigonal curves.}
\label{T:1}
\end{table}

\begin{table}
\begin{tabular}{|l|c|c|c|r|}\hline
Red. group & Automorphism group & ID & Genus & Signature\\ \hline
$A_4$ & $C_3\times A_4$& (36,11)& 12s-2, $s>0$ &  (0;2,3,3,$3^{s}$)\\ 
$A_4$ & $C_3\times A_4$& & 12s+4 &  (0;6,3,3,$3^{s}$)\\ \hline
$A_4$ & $(C_2\times C_2) \rtimes C_9$& (36,3)& 12s+6 & (0;2,9,9,$3^{s}$)\\ 
$A_4$ & $(C_2\times C_2) \rtimes C_9$& & 12s+12 & (0;6,9,9,$3^{s}$)\\ \hline
$S_4$ & $C_3\times S_4$& (72,42)& 24s-2, $s>0$ & (0;2,3,4,$3^{s}$)\\ 
$S_4$ & $C_3\times S_4$& & 24s+4 &  (0;2,3,12,$3^{s}$)\\ 
$S_4$ & $C_3\times S_4$& & 24s+16 &  (0;6,3,12,$3^{s}$)\\ 
$S_4$ & $C_3\times S_4$& & 24s+10 &  (0;6,3,4,$3^{s}$)\\ \hline
$S_4$ & $C_3\rtimes S_4$& (72,43)& 24s-2, $s>0$ &  (0;6,3,12,$3^{s}$)\\ \hline
$S_4$ & $((C_2\times C_2) \rtimes C_9)\rtimes C_2$& (72,15)& 24s+6 & (0;2,9,4,$3^{s}$)\\ \hline
$A_5$ & $C_3\times A_5$& (180,19)& 60s-2, $s> 0$ & (0;2,3,5,$3^{s}$)\\ 
$A_5$ & $C_3\times A_5$& & 60s+10 & (0;2,3,15,$3^{s}$)\\ 
$A_5$ & $C_3\times A_5$& & 60s+40 & (0;6,3,15,$3^{s}$)\\ 
$A_5$ & $C_3\times A_5$& & 60s+28 & (0;6,3,5,$3^{s}$)\\ \hline
\end{tabular}

\vskip6pt
\caption{Full automorphism group for normal trigonal curves}\label{T:2}
\end{table}

\vskip6pt
\noindent{\bf Genus $2$}

Let $X_2$ be the 3-gonal curve of genus two admitting the action of $GL(2,3)= \langle a,b:a^{8}=b^3=(ab)^2=ba^{-3}ba^{-3}=1\langle$ (see the first row in the Table \ref{T:1}). This curve is known as the Bolza curve with equation $y^2=x(x^4-1)$. The generating vector of the action is $(a,b,(ab))$ of type $(0;8,3,2)$.
\vskip6pt
\noindent{\bf Genus $3$}

The genus $3$ surface in Table \ref{T:1}, corresponds to the curve with plane model $y^4=x^3-1$ (see \cite[Table 2]{magaard}) with action of $SL(2,3)/CD$. We consider the presentation $SL(2,3)/CD= \langle a,b:a^{12}=b^3=(ab)^2=a^{11}b^{-1}aba^{-1}ba^{-7}=1\rangle$. The generating vector of the action is $(a,b,ab)$ of type $(0;12,3,2)$.  
\vskip6pt
\noindent{\bf Genus $4$}

The genus $4$ surfaces in Table \ref{T:1}, correspond to the case studied in \cite{accola1} and \cite{mila-strata}. The surfaces admit four actions of $C_3$; two conjugate with quotients of genus $0$ and two non conjugate with quotients of genus $2$. Moreover, the one dimensional locus, corresponding to surfaces of genus $4$ with automorphism group $D_3\times D_3$ consists of the curves with equation $ax^3y^3-(x^3+y^3)+a=0$, where $a\notin \{0, \pm 1, \infty\}$. This family contains the surface with action of $(C_3\times C_3)\rtimes D_4$ (see \cite[Table 4]{magaard}).  
Then, we use the action of $D_3\times D_3$. We consider the presentation $D_3\times D_3=\langle a,b,c|a^3=b^2=c^2=(abc)^2=a^2ca^{-1}bca^{-1}b^{-1}a^{-2}=a^2ca^{-1}bab^{-1}ac^{-1}a^{-1}=1\rangle$. The generating vector of the action is $(a,b,c,abc)$ of type $(0;3,2,2,2)$.

On the other hand, with respect to the groups in Table \ref{T:2}, we have that
$C_3\times A_4$ and $C_3\times S_4$ act on genus $4$. The group $C_3\times A_4$ is contained in $C_3\times S_4$ and we study the possible actions of both groups on this genus. Using \cite{magma} and \cite{brr}, we note that they act over the same surface given by the planar model $y^3=x(x^4-1)$ (\cite{pinto}). This result completes the one remaining case left open in \cite{accola1} which was studied later in \cite{magaard} and \cite{mila-strata}. Our result coincide with they obtained. The action of $A_4\times C_3$ (see Table \ref{T:2}) extends to the action of the group $S_4\times C_3=\langle a,b|a^{12}=b^3=(ab)^2=a^{11}ba^{-1}bab^{-1}ab^{-1}a^{-6}=1\rangle$ with generating vector $(a,b,ab)$ of type $(0;12,3,2)$ (see \cite{mila-strata}). 
\vskip6pt
\noindent{\bf Genus $6$}

The groups $((C_2\times C_2)\rtimes C_9)\rtimes C_2$ and $(C_2\times C_2)\rtimes C_9$ act on curves of genus $6$. The group $(C_2\times C_2)\rtimes C_9$ is contained in $((C_2\times C_2)\rtimes C_9)\rtimes C_2$ and we study their possible actions on this genus using \cite{magma} and \cite{brr}. Both groups act on the same trigonal curve \cite{magaard} with  planar model $y^3=x^8+14x^4+1$ \cite{shaska}. The group $((C_2\times C_2)\rtimes C_9)\rtimes C_2=\langle a,b|a^2=b^9=(ab)^{4}=ab^{3}ab^{3}=1\rangle$ acts on the curve with generating vector $(a,b,(ab)^{-1})$ of type $(0;2,9,4)$.

\begin{remark}
From Table \ref{T:2}, we know that there do not exist trigonal curves of genus 8 with reduced group $A_4, S_4$ and $A_5$.   
\end{remark}

\subsection{Application of the method to trigonal curves}\label{SS:results}

In this section we apply our method to trigonal curves. Let $E_j$ denotes an elliptic curve and $|\rm Kernel|$ denotes the smallest possible order for the kernel of $\nu_\times$, the isogeny defined in \ref{D:isogenyLeslie}.

\begin{table}\label{Tb-2}
\begin{tabular}{|l|c|c|c|c|c|c|c|c|c|r|}\hline

Order of the elements &  1 & 2 & 2 & 3 & 4 & 6  & 8 &  8\\ \hline
$V_1$   &  1 & 1 & 1 & 1 & 1 & 1  &1 &  1\\
$V_2$   &  1 & 1 &-1 & 1 & 1 & 1  &-1  &-1\\
$V_3$   &  2 & 2 & 0 &-1&  2 &-1 &  0 &  0\\
$V_4$   &  2& -2 & 0& -1 & 0 & 1 & $i\sqrt{2}$ & $-i\sqrt{2}$ \\
$V_5$   &  2& -2 & 0 &-1 & 0 & 1 &$-i\sqrt{2}$  & $i\sqrt{2}$ \\
$V_6$   &  3 & 3 & 1 & 0 &-1 & 0 & -1 & -1\\
$V_7$   &  3&  3 &-1 & 0 &-1 & 0 &1 &  1\\
$V_8$   &  4 &-4 & 0 & 1 & 0 &-1  & 0  & 0\\ \hline
\end{tabular}


\caption{Character Table of $GL(2,3)$.}
\end{table}

\begin{theorem}\label{T:3gonal}
Let $JX$ be the Jacobian variety of a trigonal curve $X$, of one of the types detailed below. Then $JX$ is completely decomposable, and a geometrical description of the decomposition is in the following tables.

\begin{itemize}
\item If $JX$ is the Jacobian variety of a non-normal trigonal curve, then we have the following results:

\vskip12pt

\rm 
\begin{center}
{\small
\hspace{-1cm}
\begin{tabular}{|c|c|c|c|c|c|c|}\hline
Red. & Automorphism & Genus & Decomposition & $|$Kernel$|$& Induced \\ 
group & group & of $JX$ & of $\nu_{\times}$ & & polarization\\ \hline
$D_2$ & $GL(2,3)$ &$2$&$E_1\times E_2$ & 1&$(2)$\\
$C_4$ & $SL(2,3)/CD$ &$3$&$E_1\times E_2\times E_3$ &4&$(2)$\\
$D_3$ & $D_3\times D_3$ &$4$&$E_1\times \dots \times E_4$ &9&$(2), (6), (2), (6)$\\
$D_6$ & $(C_3\times C_3)\rtimes D_4$ &$4$&$E_1\times \dots \times E_4$ &9&$(2)$\\ \hline
\end{tabular}
}
\end{center}


\item \textit{If $JX$ is the Jacobian variety of a trigonal curve with reduced group $A_4, S_4$ or $A_5$, then we have the following results:}

\vskip6pt

\rm 
{\small
\hspace{-0.9cm}
\begin{tabular}{|c|c|c|c|c|c|c|}\hline
Red. & Automorphism & Genus & Decomposition & $|$Kernel$|$& Induced \\ 
group & group & of $JX$ & of $\nu_{\times}$ & & polarization\\ \hline
$A_4$ & $C_3\times A_4$& 4 & $E_1\times E_2\times E_3\times E_4$ & 64& $(4), (3), (3), (3)$\\ 
$S_4$ & $C_3\times S_4$& 4 & $E_1\times E_2\times E_3\times E_4$ & 16& $(4)$\\ 
$S_4$ & $((C_2\times C_2)\rtimes C_9)\rtimes C_2$& 6 & $E_1\times \dots \times E_6$ & 64& $(4)$\\ \hline
\end{tabular}
}
\end{itemize}

\end{theorem}
\vspace{0.2cm}

\begin{proof}
We will show here the techniques applied to the group $GL(2,3)$ in order to limit the size of the matrices. The rest of the cases are proved following the same procedure. The reader is referred to \cite{yo} for explicit calculations.

We use our method, presented in Chapter 2. All the computations were made in the software package MAGMA \cite{magma}. The subgroups we use to obtain the effective set giving the decomposition satisfy both conditions of Lemma \ref{metodo}. Hence, the factors that will define $B_\times$ in the isogeny $\nu_\times$ will correspond to Jacobian varieties of intermediate coverings. 

Let $X_2: y^2=x(x^4-1)$ be the Bolza curve described in the case of genus 2 viewed before. A generating vector for this action is $(a,b,(ab))$ of type $(0;8,3,2)$. Using equation (\ref{langerecillas}) we obtain that the Jacobian variety $JX$ associated to $X$ is completely decomposable i.e. isogenous to a product of elliptic curves. In fact, the group algebra decomposition isogeny is $\nu: B_{41}\times B_{42}\sim JX,$ where the product $B_{41}\times B_{42}$ is invariant by the irreducible rational representation $V_4$ of degree 2 (see Table 3). 

To apply our method to find an explicit decomposition $\nu_{\times}$, we first look for two subgroups $H_1$ and $H_2$ of $GL(2,3)$ satisfying conditions of Lemma \ref{metodo}. To find these subgroups, we use the symplectic representation of the action obtained from the method given in \cite{brr}. This allows us to write every element of $GL(2,3)$ as a square symplectic matrix of size 4.  

We determine next the complex irreducible representation decomposition of the induced representation $\text{Ind}_H^G 1_H$ in $GL(2,3)$ by the trivial representation in $H$ for each $H\leq G$ (See Table 4). We know that this decomposition of $\text{Ind}_H^G 1_H$ into $\CC-$irreducible representations is invariant under conjugation (see Remark \ref{hconjugate}). Table \ref{ind-2} shows the multiplicity of each complex irreducible representation in the induced representation  $\text{Ind}_H^G 1_H$, for all $H\leq G$ up to conjugacy.

\begin{table}\label{ind-2}
\begin{tabular}{|l|c|c|c|c|c|c|c|c|c|c|r|}\hline
Classes of subgroups & $V_1$ &$V_2$ &$V_3$ &$V_4$ &$V_5$ &$V_6$ &$V_7$ &$V_8$\\ \hline
Identity element& 1& 1& 2& 2& 2& 3& 3& 4\\
order 2, length 1& 1& 1& 2& 0& 0& 3& 3& 0\\
order 2, length 12&1& 0& 1& 1& 1& 2& 1& 2\\
order 3, length 4&1& 1& 0& 0& 0& 1& 1& 2\\
order 4, length 3&1& 1& 2& 0& 0& 1& 1& 0\\
order 4, length 6&1& 0& 1& 0& 0& 2& 1& 0\\
order 6, length 4&1&1& 0& 0& 0&1& 1& 0\\
order 6, length 4&1& 0& 0& 0& 0& 1& 0& 1\\
order 6, length 4&1& 0& 0& 0& 0& 1& 0& 1\\
order 8, length 1&1& 1& 2& 0& 0& 0& 0& 0\\
order 8, length 3&1& 0& 1& 0& 0& 0& 1& 0\\
order 8, length 3&1& 0& 1& 0& 0&1& 0& 0\\
order 12, length 4&1& 0& 0& 0& 0& 1& 0& 0\\
order 16, length 3&1& 0& 1& 0& 0&0& 0& 0\\
order 24, length 1&1& 1&0& 0& 0&0& 0& 0\\
order 48, length 1&1& 0& 0& 0& 0& 0& 0& 0\\ \hline
\end{tabular}
\caption{Decomposition of the induced representation by the trivial one on each class of conjugation of subgroups of $GL(2,3)$.}
\end{table}

We observe that the only conjugacy class of subgroups of $GL(2,3)$ whose elements satisfy the conditions in Lemma \ref{metodo} consists of the class of $H=\langle ab\rangle$. For each subgroup $H$ in this class 
$$\langle \text{Ind}_H^G 1_H,V_4\rangle=1.$$


\noindent By Lemma \ref{metodo}, each factor in the decomposition of $JX$ is defined by $B_{4j} = \ima(p_{H_j}) = V_{H_j}/\Lambda_{H_j}$, where $H_j$ is some subgroup in the class and $j\in \{1,2\}$. We obtain 
$\nu_{\times}:B_{41}\times B_{42}\to J X$, and its kernel depends on the choice of $H_j$ in this class.

To give an explicit description of the subgroups of $GL(2,3)$ giving the smallest order for the kernel of $\nu_{\times}$, we consider the above presentation of $GL(2,3)$ and the same generating vector $(a,b,ab)$. Set $H=\langle ab\rangle$, and consider the following subgroups of order 2 in $GL(2,3)$
\begin{eqnarray*}
H_1&=& H^{b}=\;\langle ba^7b\rangle,\\
H_2&=& H=\;\langle ab\rangle,
\end{eqnarray*}
and write $B_{41} = \ima(p_{H_1})$ and $B_{42} = \ima(p_{H_2})$. Since 

$$\rho_s(p_{H_1})= \left(\begin{array}{cccc} 0 &  0 & 0 & 0\\
1/2 & 1 & 0 & 0\\
0 & 0 & 0 & 1/2\\
0 & 0 & 0 & 1\end{array}\right), \; \rho_s(p_{H_2})= \left(\begin{array}{cccc} 1/2 & 1/2 & 0 & 0\\
1/2 & 1/2 & 0 & 0\\
0 & 0 & 1/2 & 1/2\\
0 & 0 & 1/2 & 1/2\end{array}\right),$$
and the coordinate matrices of their lattices are
$$\Lambda_{H_1}= \left(\begin{array}{cccc} 0 & 1 & 0 & 0\\
0 & 0 & 1 & 2\end{array}\right), \;\Lambda_{H_2}= \left(\begin{array}{cccc} 1 & 1 & 0 & 0\\
0 & 0 & 1 & 1\end{array}\right),$$
we find that the matrix coordinate of the lattice of the product $B_{41}\times B_{42}$ is given by

$$L_{\{H_1,H_2\}}= \left(\begin{array}{cccc} 0 & 1 & 0 & 0\\
0 & 0 & 1 & 2\\
1 & 1 & 0 & 0\\
0 & 0 & 1 & 1\end{array}\right).$$
In this case, choosing $H_1$ and $H_2$ as before $\nu_{\times}:B_{41}\times B_{42}\to JX$ is an isomorphism. Hence $|\text{Ker}(\nu)| = 1$. Note that we do not claim that the subgroups yielding an isomorphism are unique. In fact, there exist other subgroups in the same class such that the $|\text{Ker}(\nu_{\times})|=1$. 
\end{proof}

\begin{remark}\label{moduli}
Finally, combining Theorem \ref{T:3gonal} and the classification of \cite{magaard}, we may describe part of the loci of Jacobians of trigonal curves for dimension $g\leq 6$.
\\
\rm
\begin{center}
\begin{tabular}{|c|c|c|}\hline
$g$ & Dimension of the family & Locus description \\ \hline
$2$ & 0& one curve with action of $\text{GL}(2,3)$ \\ 
$3$ & 0& one curve with action of $\text{SL}(2,3)/\text{CD}$ \\ 
$4$ & 1& 1-dimensional with action of $D_3\times D_3$ \\
& & and one curve with action of $C_3\times S_4$ \\ 
$6$ & 0& one curve with action of $((C_2\times C_2)\rtimes C_9)\rtimes C_2$ \\ \hline
\end{tabular}
\end{center}
\vspace{0.2cm}

\end{remark}

\begin{acknowledgements}
This article is part of my PhD. thesis written under the direction of Professor Anita Rojas at the Universidad de Chile. I am very grateful to Professor Rojas for sharing, patiently and kindly, her knowledge, experiences and advice. I would like to thank Professors A. Carocca, R. Rodr\'iguez for helpful advice and questions, and M. Izquierdo, who helped me to present this work. I also express my acknowledgments to Linkoping University, where the final version of this paper was written. 
\end{acknowledgements}



\end{document}